\documentclass[12pt]{amsart}
\usepackage{latexsym} 
\usepackage{amsmath,amssymb}
\usepackage{verbatim}
\usepackage[dvips]{graphicx}
\newtheorem{teo}{Theorem}

\newtheorem{remark}{Remark}
\newtheorem{proposition}{Proposition}

\newtheorem{lemma}{Lemma}

\theoremstyle{definition}
\newtheorem{defn}{Definition}

\begin{document}

\title[Distributional fractional powers of similar operators]{
Distributional fractional powers of similar operators.  Applications to the Bessel operators}
\begin{abstract}
This paper provides a method to study the non-negativity of certain linear operators, from other  operators with similar spectral properties. If these
new operators are formally self-adjoint and non-negative, we can study the complex powers using an  appropriate locally convex space. In this case, the initial operator
also will be non-negative and we will be able to study their powers. In
particular, we have applied this method to Bessel-type operators.
\end{abstract}

\author[S. Molina]{Sandra Molina}
\address{Departamento de Matem\'{a}ticas. Facultad de Ciencias Exactas y Naturales\\Universidad Nacional de Mar del Plata\\
Funes 3350 (7600)\\
Mar del Plata, Argentina} \email{smolina@mdp.edu.ar}

\maketitle

\section{Introduction}
Operators of type Bessel appear in the literature related with different versions of Hankel transform (see \cite{al},\cite{HI},\cite{MB},\cite{az}).  We are going to consider   Bessel operators on $(0,\infty)$  given by

\begin{equation}\label{besselOp1}
\Delta _{\mu }=\frac{d^{2}}{dx^{2}}+(2\mu +1)(x^{-1}\frac{d}{dx})=x^{-2\mu -1}\frac{d}{dx}x^{2\mu +1}\frac{d}{dx}
\end{equation}
and
\begin{equation}\label{besselOp2}
S_{\mu}=\frac{d^{2}}{dx^{2}}-\frac{4\mu^{2}-1}{4x^{2}}=x^{-\mu -\frac{1}{2}}\frac{d}{dx}x^{2\mu +1}\frac{d}{dx}x^{-\mu -\frac{1}{2}},
\end{equation}
wich are related through
$$S_{\mu}=x^{\mu +\frac{1}{2}}\Delta _{\mu }x^{-\mu -\frac{1}{2}}.$$
This feature has inspired us to develop a method to study fractional powers based in a concept of "similar operator". Similar operators  have the same spectral properties between them and of being non negative if one of them has this property. This method will apply in the contexts of  Banach spaces and  locally convex spaces.

Moreover, other  feature of  operators (\ref{besselOp1}) and (\ref{besselOp2}) is that one of these is selfadjoint and the other is not. To establish the complex powers of a differential operator in distributional spaces is important that this operator be formally self-adjoint. It would therefore be
interesting to obtain an operator similar and formally self-adjoint from a given initial operator.

To study the distributional fractional powers of Bessel operators, we use the theory developed by Martinez y Sanz in \cite{ms}.
\vskip.2in

In section 2 we will review some of the standard facts about Hankel transform,  convolution and Bessel operator in distributional and Lebesgue spaces, which are fundamental to establish the non negativity of Bessel operator.
In Section 3 we will state our main results about  similar operators in Banach spaces and we will describe the relation between their  fractional powers. We extend this idea to locally convex spaces and we apply this ideas to the Bessel operator (\ref{besselOp1}) and (\ref{besselOp2}).

In section 4  we will establish the non-negativity of $S_{\mu }$ in suitable weighted Lebesgue spaces. In  sections 5 and 6 we will establish the non negativity of $S_{\mu }$ in a suitable locally convex space and in its dual space.

For the convenience of the reader, we have added an Appendix with the proofs of some results about Hankel transform, convolutions, thus making our exposition self-contained.

\vskip.3in

\section{Previous results on Hankel transform and convolution}
In this section we introduce the Lebesgue and distributional spaces  necessary for our purposes.

\noindent  By $\mathcal{D}(0,\infty)$ we denote the space of functions in
 $C^{\infty }(0,\infty)$ with compact support in $(0,\infty)$ with the usual topology,  and by $\mathcal{D'}(0,\infty)$ the space of
  classical distributions in $(0,\infty)$.

Throughout this paper we assume $\mu > -\frac{1}{2}$. We will consider the Hankel transform defined in a suitable
functional space denoted by $\mathcal{H}_{\mu }$ and given by

\begin{equation*}
\mathcal{H}_{\mu }=\left\{ \phi \in C^{\infty}(0,\infty):
\underset{x\in (0,\infty )}{\sup }\left| x^{m}(x^{-1}D)^{k}x^{-\mu -\frac{1}{%
2}}\phi (x)\right| <\infty :m,k=0,1,2,...\right\}
\end{equation*}
endowed with the family of seminorms $\left\{ \gamma _{m,k}^{\mu }\right\},$ given by

\begin{equation}\label{eqsemihmu}
\gamma _{m,k}^{\mu }(\phi )=\underset{x\in (0,\infty )}{\sup }\left| x^{m}(x^{-1}D)^{k}x^{-\mu -\frac{1}{%
2}}\phi (x)\right|,
\end{equation}
$\mathcal{H}_{\mu }$ is a  Fr\'{e}chet space (see \cite[Lemma 5.2-2, pp.
131]{az}).
 Given  $1\leq p <\infty$ and a measurable function $w:(0,\infty) \rightarrow \Bbb{C}$  then we consider

\begin{equation}\label{defespLp}
L^{p}((0,\infty), w\:dx)=\left\{ f:(0,\infty) \rightarrow \Bbb{C}:\: f
\:\text{is measurable and}\:\int_{0}^{\infty}\left| f(x)\right|
^{p}w(x)dx<\infty \right\}
\end{equation}
where with $dx$ we denote the Lebesgue usual measure. In $L^{p}((0,\infty), w\:dx)$ we consider the usual norm
\begin{equation*}
\left\| f\right\| _{L^{p}((0,\infty), w\:dx)}=\left[ \int_{0}^{\infty}\left|
f(x)\right| ^{p}w(x)dx\right] ^{1/p}.
\end{equation*}

\noindent Moreover,
$$L^{\infty }((0,\infty),w)=\left\{ f:(0,\infty) \rightarrow \Bbb{C}:\:\:\text{measurable and}\qquad \text{ess sup}_{x\in (0,\infty)}|wf(x)|<\infty \right\}$$
endowed with the norm
\begin{equation*}
\left\| f\right\| _{L^{\infty }((0,\infty),w)}=\left\| wf\right\|_{\infty }.
\end{equation*}
For simplicity of notation we write $L^{p}(w)$ and $L^{\infty}(w)$ instead of $L^{p}((0,\infty),w\:dx)$ and $L^{\infty}((0,\infty),w)$.

Let  $s$ and $r$  as follow
\begin{equation}\label{eqr}
r=x^{-\mu -\frac{1}{2}}
\end{equation}
\begin{equation}\label{eqs}
s=x^{2\mu +1}/c_{\mu }
\end{equation}
with $c_{\mu }=2^{\mu }\Gamma (\mu +1)$.
\begin{proposition}
\label{espacio hm} It is verified that
\begin{equation}\label{inclusHm}
\mathcal{H}_{\mu}\subset L^{1}(sr)\cap L^{\infty }(r)\subset L^{p}(sr^{p}) \:\:\:1\leq p<\infty,
\end{equation}
with $r$ and $s$ given by (\ref{eqr}) and (\ref{eqs}).
\end{proposition}

\begin{proof}
 The inclusion $\mathcal{H}_{\mu}\subset L^{\infty}(r)$ is
immediate and also

\begin{equation}  \label{acotainf}
\left\| \phi \right\| _{L^{\infty }(r)}= \gamma _{0,0}^{\mu}(\phi),\quad \phi \in \mathcal{H}_{\mu}.
\end{equation}
It also verifies that $\mathcal{H}_{\mu} \subset L^{1}(sr)$  as
$$\int_{0}^{\infty}|\phi| \: sr\: dx=\int_{0}^{1}|x^{-\mu-\frac{1}{2}}\phi|x^{2\mu+1}c_{\mu}^{-1}\:dx+
\int_{1}^{\infty}x^{m}|x^{-\mu-\frac{1}{2}}\phi|x^{-m+2\mu+1}c_{\mu}^{-1}\:dx<\infty$$
if $m>2\mu+2$, and
\begin{equation}\label{acota L1}
 \| \phi \| _{L^{1}(sr)}\leq C \{\gamma _{0,0}^{\mu}(\phi)+\gamma_{m,0}^{\mu}(\phi)\} ,\quad \phi \in \mathcal{H}_{\mu}
\end{equation}

It also verifies that
\begin{equation}\label{acota Lp}
\left\| \phi \right\| _{L^{p}(sr^{p})}=\Bigl \{ \int_{0}^{\infty} |\phi|^{p-1}r^{p-1} |\phi|rs \Bigr \}^{\frac{1}{p}}\leq \Bigl \{ \left\| \phi \right\| _{L^{\infty}(r)}\Bigr \}^{\frac{p-1}{p}}\Bigl \{ \left\| \phi \right\| _{L^{1}(sr)}\Bigr \}^{\frac{1}{p}}
\end{equation}
and by  (\ref{acotainf}) and (\ref{acota L1}) we can consider a constant  $C'$ such that

\begin{equation} \label{acotainf1}
\left\| \phi \right\| _{L^{\infty }(r)}\leq C'\left[ \gamma
_{0,0}^{\mu }(\phi )+\gamma _{m,0}^{\mu }(\phi )\right],\quad \phi
\in \mathcal{H}_{\mu }.
\end{equation}

\begin{equation}\label{acota L11}
\left\| \phi \right\| _{L^{1}(sr)}\leq C'\left[ \gamma _{0,0}^{\mu }(\phi
)+\gamma _{m,0}^{\mu }(\phi )\right] ,\quad \phi \in \mathcal{H}_{\mu }
\end{equation}
and by (\ref{acota Lp}), (\ref{acotainf1}) and  (\ref{acota L11}) we finally  conclude that
\begin{equation}\label{acotatotal}
\left\| \phi \right\| _{L^{p}(sr^{p})}\leq C'\left[\gamma_{0,0}^{\mu}(\phi)+\gamma_{m,0}^{\mu}(\phi)\right],\quad \phi\in\mathcal{H}_{\mu}
\end{equation}
\end{proof}
\vskip.2in

If $J_{\mu}$ denote the Bessel function of first kind and order $\mu$, we consider the Hankel
transform  $\mathit{h}_{\mu }$ given by
\begin{equation}
\mathit{h}_{\mu}\phi(x)=\int_{0}^{\infty}\sqrt{xy}J_{\mu}(xy)\phi
(y)dy. \label{transf Hankel}
\end{equation}
for $\phi \in \mathcal{H}_{\mu }$.

\begin{remark}
\label{notatra hankel} If $\phi \in L^{1}(sr)$ then Hankel transform
$\mathit{h}_{\mu }\phi $ is well defined because the kernel
$(xy)^{-\mu }J_{\mu }(xy)$ is bounded if $\mu > -\frac{1}{2}$ (see \cite[(1), pp. 49]{wa}) .
By Proposition \ref{espacio hm}, $\mathit{h}_{\mu}\phi$   is well
defined for all $\phi \in \mathcal{H}_{\mu}$ and  is an automorphism (see  \cite[theorem 5.4-1, pp. 141]{az}).
\end{remark}

The space of the continuous linear functions $T:\mathcal{H}_{\mu }\rightarrow \mathbb{C}$ is denoted by
$\mathcal{H}_{\mu }^{\prime }$.



\begin{defn}
\label{regular h' mu} We call a function $f\in L_{loc}^{1}((0,\infty) )$ a {\it regular element} of
$\mathcal{H}_{\mu }^{\prime }$ if the identification
\begin{equation*}
\begin{array}{llll}
T_{f}: & \mathcal{H}_{\mu } & \longrightarrow  & \Bbb{C} \\
& \phi  & \longrightarrow  & (T_{f},\phi )=\int_{0}^{\infty }f\phi
\end{array}
\end{equation*}
is well defined and is continuous; namely  $T_{f}\in \mathcal{H}_{\mu }^{\prime }$.
\end{defn}
\begin{remark}\label{regdistrib} Given $T\in \mathcal{H}_{\mu }^{\prime }$, we can consider the restriction of $T$ to $\mathcal{D}(0,\infty)$ as a member of $\mathcal{D}^{\prime}(0,\infty)$, because convergence in $\mathcal{D}(0,\infty)$ implies convergence in $\mathcal{H}_{\mu }$. But $\mathcal{D}(0,\infty)$  is not  dense in $\mathcal{H}_{\mu }$ (see \cite{az}), consequently the behavior of an
element $u\in \mathcal{H}_{\mu }^{\prime}$ over
$\mathcal{D}(0,\infty)$ not determines univocally the behavior of
$u$ as element of $\mathcal{H}_{\mu }^{\prime }.$  If a locally integrable function is zero as regular element of  $\mathcal{H}_{\mu }^{\prime }$, then is zero  a.e in $(0,\infty)$ because is zero as regular distribution in $\mathcal{D}^{\prime}(0,\infty)$. So,  regular distributions in $\mathcal{H}_{\mu }^{\prime}$ are included in injective way in $\mathcal{D'}(0,\infty)$.
\end{remark}
\begin{proposition}
Suppose that $1\leq p<\infty $. A function in $L^{p}(sr^{p})$ or  $L^{\infty }(r)$ is a regular element of $\mathcal{H}_{\mu }^{\prime }$.
In particular, the functions in $\mathcal{H}_{\mu }$ can be considered as regular elements of $\mathcal{H}_{\mu}^{\prime }$.
\end{proposition}

\begin{proof}
Let  $f\in L^{\infty }(r)$,  since $\mathcal{H}_{\mu }\subset
L^{1}(sr)=L^{1}(r^{-1}/c_{\mu })$,  if $\phi \in \mathcal{H}_{\mu
}$,  then $\phi \in L^{1}(r^{-1})$ and $(T_{f},\phi
)=\int_{0}^{\infty }f\phi $ is well defined for $\phi \in \mathcal{H}%
_{\mu }$. So, by (\ref{acota L1})
$$
\left| (T_{f},\phi )\right| \leq \left\| f\right\| _{L^{\infty }(r)}\left\|
\phi \right\| _{L^{1}(r^{-1})}=c_{\mu }\left\| f\right\| _{L^{\infty
}(r)}\left\| \phi \right\| _{L^{1}(sr)}$$
$$\leq Cc_{\mu }\left\| f\right\|
_{L^{\infty }(r)}\left[ \gamma _{0,0}^{\mu }(\phi )+\gamma _{m,0}^{\mu
}(\phi )\right]$$

\noindent consequently,  $f$ is a regular element of $\mathcal{H}_{\mu }^{\prime }$.

\noindent Now, let $f\in L^{p}(sr^{p})$ with $1\leq p<\infty$, then
\begin{equation}
\left| (T_{f},\phi )\right| \leq \int_{0}^{\infty }\left| f\phi \right|
=\int_{0}^{\infty }\left( r\left| f\right| \right) \left( s^{-1}r^{-1}\left|
\phi \right| \right) s=\int_{0}^{\infty }\left( r\left| f\right| \right)
\left( c_{\mu }r\left| \phi \right| \right) s  \label{identif f}
\end{equation}
Since $r\left| f\right| \in L^{p}(s)$ and $r\left| \phi \right| \in
L^{q}(s)$ because $\phi \in \mathcal{H}_{\mu }\subset L^{q}(sr^{q})$ (see (\ref{inclusHm})), being  $q$ the conjugate of $p$, we obtain by
 H\"{o}lder inequality and (\ref{acotatotal})
\begin{equation*}
\left| (T_{f},\phi )\right| \leq c_{\mu } \left\| f \right\|
_{L^{p}(sr^{p})}\left\| \phi \right\| _{L^{q}(sr^{q})}\leq Cc_{\mu }\left\|
f\right\| _{L^{p}(sr^{p})}\left[ \gamma _{0,0}^{\mu }(\phi )+\gamma
_{m,0}^{\mu }(\phi )\right]
\end{equation*}
with $m$ a positive integer $m>2\mu +2$. So, $f$ is a regular element of $\mathcal{H}_{\mu }^{\prime }$.
\end{proof}
\vskip.2in

Given $f,g$ defined in $(0,\infty )$, the Hankel convolution $f\sharp g$ is defined formally by
\begin{equation}\label{conv zem}
\left( f\sharp g\right) (x)=\int_{0}^{\infty }\int_{0}^{\infty
}D_{\mu }(x,y,z)f(y)g(z)\:dydz
\end{equation}
where $D_{\mu }(x,y,z)$ is given by

\begin{equation}
D_{\mu }(x,y,z)=\left \{
\begin{array}{ccc}
\frac{2^{\mu-1}(xyz)^{-\mu+\frac{1}{2}}}{\Gamma(\mu+\frac{1}{2})\sqrt{\pi}}(A(x,y,z))^{2\mu-1} &\mbox{si}&\vert x-y \vert <z <x+ y \\
0 &\mbox{si}& 0<z<\vert x-y \vert \:\: \mbox{o}\:\: z> x+y .\\
\end{array} \right.
\end{equation}

\noindent   $A(x,y,z)$ is the measure of area of the triangle with sides $x,y,z\quad$  and    $\quad\vert x-y \vert <z <x+ y$ is the condition  for such a triangle to exist, and in this case  $A(x,y,z)=\frac{1}{4}\sqrt{[(x+y)^{2}-z^{2}][z^{2}-(x-y)^{2}]}$.

\vskip.2in
The proofs of Theorems \ref{teoyoung} and \ref{approxiden} and Proposition \ref{prop hankel conv} are displayed in Appendix.
\begin{teo}\label{teoyoung}
Let $f\in L^{1}(sr)$.

\begin{enumerate}
\item  If $g\in L^{\infty }(r)$,  then the convolution   $f\sharp g(x)$ exists for every $x\in (0,\infty )$,
$f\sharp g \in
L^{\infty }(r)$ and
\begin{equation}\label{young Linfty}
\left\| f\sharp g\right\| _{L^{\infty }(r)}\leq \left\| f\right\|
_{L^{1}(sr)}\left\| g\right\| _{L^{\infty }(r)}.
\end{equation}

\item  If $g\in L^{p}(sr^{p})$ $(1\leq p<\infty )$,
then the convolution $f\sharp g(x)$ exists for a. e. $x\in (0,\infty ),$ $f\sharp g$ $\in L^{p}(sr^{p})$  and
\begin{equation}\label{young lp}
\left\| f\sharp g\right\| _{L^{p}(sr^{p})}\leq \left\| f\right\|
_{L^{1}(sr)}\left\| g\right\| _{L^{p}(sr^{p})}.
\end{equation}
\end{enumerate}
\end{teo}
\vskip.2in

\begin{teo}\label{approxiden} Let $\{\phi_{n}\}\subset L^{1}(rs)$ such that
\begin{enumerate}
\item $\phi_{n}\geq 0$ in $(0,\infty)$,
\item $\int_{0}^{\infty}\phi_{n}(x) r(x)s(x)\:dx=1$  for all n
\item For $\delta > 0$ , $\lim_{n\rightarrow \infty}\int_{\delta}^{\infty} \phi_{n}(x)r(x)s(x)\:dx=0$.
\end{enumerate}
\noindent Let $f\in L^{\infty}(r)$ and continuous in $x_{0}\in(0,\infty)$, then $\lim_{n\rightarrow \infty} f\sharp \phi_{n}(x_{0})=f(x_{0})$.
Further, if   $\:\:rf\:\:$   is uniformly continuous in $(0,\infty)$ then  $\lim_{n\rightarrow \infty} \|f\sharp \phi_{n}(x)-f(x)\|_{L^{\infty}(r)}=0$
\end{teo}
\vskip.2in

\begin{proposition}\label{prop hankel conv}
Let $f,g\in L^{1}(sr)$, then
\begin{equation}
\mathit{h}_{\mu }\left( f\sharp g\right) =r\mathit{h}_{\mu }(f)\mathit{h}%
_{\mu }(g).  \label{hankel convol}
\end{equation}
\end{proposition}

\vskip.3in

\subsection{The Bessel operator $S_{\mu }$}

In this section we summarize some elementary properties of $S_{\mu }$ on the spaces $\mathcal{H}_{\mu }$ and $\mathcal{H}'_{\mu }$. For most of the proofs we refer the reader to \cite{az}.

\begin{proposition}
\label{notas smu}

\begin{enumerate}
\item  The operator $S_{\mu }:\mathcal{H}_{\mu } \longrightarrow  \mathcal{H}_{\mu }$  is continuous.

\item  If $\lambda \geq 0$,  the operator
\begin{equation*}
\begin{array}{lll}
\mathcal{H}_{\mu } & \longrightarrow  & \mathcal{H}_{\mu } \\
\phi  & \longrightarrow  & (\lambda +x^{2})\phi
\end{array}
\end{equation*}
is continuous.

\item  If $\lambda >0$,  the operator
\begin{equation*}
\begin{array}{lll}
\mathcal{H}_{\mu } & \longrightarrow  & \mathcal{H}_{\mu } \\
\phi  & \longrightarrow  & (\lambda +x^{2})^{-1}\phi
\end{array}
\end{equation*}
is continuous.
\end{enumerate}
\end{proposition}

\begin{proposition}
\label{h mu s mu} Let $\phi \in \mathcal{H}_{\mu }$, then

\begin{enumerate}
\item  $\left( \mathit{h}_{\mu }S_{\mu }\phi \right) (x)=-x^{2}\left(\mathit{h}_{\mu}\phi \right) (x),\quad x\in (0,\infty ).$

\item  $\left( S_{\mu }\mathit{h}_{\mu }\phi \right) =\mathit{h}_{\mu}(-y^{2}\phi (y)).$
\end{enumerate}
\end{proposition}

\vskip.2in

\begin{proposition}\label{operhmu} The following continuous operators in $\mathcal{H}_{\mu }$ can be extended to $\mathcal{H}_{\mu }^{\prime }$  in the following way:

\begin{enumerate}
\item  The Hankel transform $\mathit{h}_{\mu }$
\begin{equation*}
(\mathit{h}_{\mu }u,\phi )=(u,\mathit{h}_{\mu }\phi ),\quad u\in \mathcal{H}_{\mu }^{\prime },\phi \in \mathcal{H}_{\mu },
\end{equation*}

 and $\mathit{h}_{\mu }:\mathcal{H}_{\mu}^{\prime }\rightarrow \mathcal{H}_{\mu }^{\prime }$  is a bijective mapping .

\item  The diferential operator  $S_{\mu }$
\begin{equation*}
(S_{\mu }u,\phi )=(u,S_{\mu }\phi ),\quad u\in \mathcal{H}_{\mu }^{\prime
},\phi \in \mathcal{H}_{\mu }.
\end{equation*}

\item  The product of $(\lambda +x^{2})$ for $\lambda \geq 0$
\begin{equation*}
((\lambda +x^{2})u,\phi )=(u,(\lambda +x^{2})\phi ),\quad u\in \mathcal{H}%
_{\mu }^{\prime },\phi \in \mathcal{H}_{\mu }.
\end{equation*}

\item  The product of $(\lambda +x^{2})^{-1}$ for $\lambda >0$
\begin{equation*}
((\lambda +x^{2})^{-1}u,\phi )=(u,(\lambda +x^{2})^{-1}\phi ),\quad u\in
\mathcal{H}_{\mu }^{\prime },\phi \in \mathcal{H}_{\mu }.
\end{equation*}
\end{enumerate}
\end{proposition}

\begin{proposition}\label{propHankelSmu} If $u\in \mathcal{H}_{\mu }^{\prime }$, then
\begin{enumerate}
\item $\mathit{h}_{\mu }S_{\mu }u =-x^{2}\mathit{h}_{\mu }u $
\item $S_{\mu }\mathit{h}_{\mu }u =\mathit{h}_{\mu }(-y^{2}u)$
\end{enumerate}
\end{proposition}

\begin{proposition}\label{prophankelsmumas1} The following equalities are valid in $\mathcal{H}_{\mu }$ and $\mathcal{H}_{\mu }^{\prime }$ for $n=1,2,\dots$ \begin{enumerate}
\item $(-S_{\mu }+\lambda)^{n}\mathit{h}_{\mu }=\mathit{h}_{\mu}(y^{2}+\lambda)^{n}$.\\

\noindent Moreover if $\lambda> 0$
\item $\mathit{h}_{\mu }(-S_{\mu }+\lambda)^{-n}=(y^{2}+\lambda)^{-n}\mathit{h}_{\mu}$.
\item $\mathit{h}_{\mu }(-S_{\mu }(-S_{\mu }+\lambda)^{-1})^{n}=y^{2n}(y^{2}+\lambda)^{-n}\mathit{h}_{\mu}$

\end{enumerate}
\end{proposition}
\begin{proof} (1)  is immediate consequence of  item (2) of  Proposition \ref{propHankelSmu} .\\
Noting that  $\mathit{h}_{\mu }(y^{2}+\lambda)^{-1}\mathit{h}_{\mu }$ is  inverse operator of $-S_{\mu }+\lambda$, then (2) is obtained with a simple aplication of Propositon \ref{invHankel} (see Appendix) and induction over n. \\
Equality (3) follows  immediately by item (2) of Proposition \ref{propHankelSmu} and induction over n.
\end{proof}
\vskip.2in
\section{Previous results about similar operators and no-negativity}

Let $X$ and $Y$ be Banach spaces. Suppose that there is an isometric isomorphism $T:X\rightarrow Y$ and let $T^{-1}:Y\rightarrow X$ its inverse. Let $A$ a linear operator $A:D(A)\subset X\rightarrow X$ then we can consider the operator $B=TAT^{-1}$, $B:D(B)\subset Y\rightarrow Y$ with domain
$$D(B)=\{x\in Y: T^{-1}x\in D(A)\},$$
Under these conditions we will say that such operators are \textit{similars}.
\begin{proposition} \label{propopsimilar} Let $A$ and $B$ similar operators. Then $A$ is non negative if and only if  so is $B$.
\end{proposition}
\begin{proof} Let $B=TA T^{-1}$ , the proof is immediate noting that
$$(zId-B)^{-1}=T(zId-A)^{-1}T^{-1},$$
for  a complex number $z$ .
\end{proof}
\vskip.2in
If $A$ is a non negative operator, then for $\alpha\in\mathbb{C}$ such that $\textrm{Re}\:\alpha>0$, $n>\textrm{Re}\:\alpha$ and $\phi\in D(A^{n})$,  the Balakrisnahn operator associated with $A$, can be represented by
$$J_{A}^{\alpha}\phi=\frac{\Gamma(m)}{\Gamma(\alpha)\Gamma(m-\alpha)}\int_{0}^{\infty}\lambda^{\alpha-1}[A(\lambda+A)^{-1}]^{m}\phi\:d\lambda,$$
(see \cite[Proposition 3.1.3, pp.59]{ms}).

 If $A$ is bounded, $J_{A}^{\alpha}$ can be consider as  fractional power of $A$, and in other case we can consider the following representation for fractional power given in \cite[Theorem 5.2.1, pp.114]{ms}),
$$A^{\alpha}=(A+\lambda)^{n}J_{A}^{\alpha}(A+\lambda)^{-n},$$
with $\alpha$, $n$ as above  and $\lambda\in\rho(-A)$.

\noindent When two operators are similar, the fractional powers also meet this property. Thus, we have the following result:

\begin{proposition} \label{propobalakrisnan}Let $A$ and $B$ similar and non negative operators. If $\alpha\in\mathbb{C}$ such that $\textrm{Re}\:\alpha>0$,  then
\begin{equation}\label{eqsimbalak}
J_{B}^{\alpha}=TJ_{A}^{\alpha}T^{-1},
\end{equation}
and
\begin{equation}\label{eqsimilpot}
B^{\alpha}=TA^{\alpha}T^{-1},
\end{equation}

\noindent where $T$ is the isometric isomorphism that verifies $B=TAT^{-1}$.
\end{proposition}
\begin{proof} We observe that if $B=TAT^{-1}$ then $B^{n}=TA^{n}T^{-1}$ and
$$D(T J_{A}^{\alpha}T^{-1})=\{x\in Y:T^{-1} x\in D(J_{A}^{\alpha})\}=\{x\in Y:T^{-1} x\in D(A^{n})\}=D(J_{B}^{\alpha}),$$

\noindent and (\ref{eqsimbalak}) is immediate from properties of Bochner integral.
In (\ref{eqsimilpot})  the equality of domains is evident and
$$B^{\alpha}=(B+\lambda)^{n}J_{B}^{\alpha}(B+\lambda)^{-n}=(TAT^{-1}+\lambda)^{n}TJ_{A}^{\alpha}T^{-1}(TAT^{-1}+\lambda)^{-n}=$$
$$=T(A+\lambda)^{n}T^{-1}TJ_{A}^{\alpha}T^{-1}T(A+\lambda)^{-n}T^{-1}=TA^{\alpha}T^{-1}.$$
\end{proof}

\vskip.2in
Now, we consider  a Hausdorff locally convex space $X$ with  a direct family of seminorms $\{\|\|_{X,\nu}\}_{\nu\in \mathcal{A}}$. Let $Y$ be a linear space such that there is a linear isomorphism $L:X\rightarrow Y$. Then, we can define the following family of seminorms in Y:
$$\|y\|_{Y,\nu}=\|L^{-1}(y)\|_{X,\nu}.$$
Thus, the linear space $Y$ with the directed family of seminorms $\{\|\|_{Y,\nu}\}_{\nu\in \mathcal{A}}$ is a Hausdorff locally convex space  and $X$ and $Y$ are isomorphic. The propositions \ref{propopsimilar} and \ref{propobalakrisnan} can be easily extended to the case of non negative operators in locally convex spaces.

\subsection{Applications to  Bessel operator}

Given  $\mu >-\frac{1}{2}$, we consider the differential operator  given by (\ref{besselOp1})  defined in $(0,\infty)$.

We are now going to apply the observations considered in the
previous section to the operator $\Delta _{\mu }$.  First, we calculate the Sturm-Liouville form of $\Delta _{\mu }$, thereby obtaining the
operator
\begin{equation*}
T_{\mu }=x^{2\mu +1}\Delta _{\mu }
\end{equation*}
which is formally self-adjoint. Operators of type $fT_{\mu }f$, with
$f\in C^{\infty }(0,\infty)$, are still formally self-adjoint. If we
want the new operator to be similar to $\Delta _{\mu }$, namely type
$r^{-1}\Delta _{\mu }r$, we have to consider $r=x^{-\mu
-\frac{1}{2}}$. Thus the operator
\begin{equation}\label{def smu}
S_{\mu }=x^{\mu +\frac{1}{2}}\:\Delta _{\mu }\:x^{-\mu
-\frac{1}{2}}
\end{equation}
is formally self-adjoint and similar to $\Delta _{\mu }$ and hence with the same espectral properties.

Since  mappings  $L_{r}:L^{p}(r^{p}s)\rightarrow L^{p}(s)$ with $1\leq p<\infty$ (or $L_{r}:L^{\infty}(r)\rightarrow L^{\infty}$)  given by $L_{r}(f)=rf$  are isometric isomorphisms, if we consider the part of the distributional operator $\Delta_{\mu }$ in the spaces $L^{p}(s)$ (or $ L^{\infty}$), i.e., the operator with domain
 $$D((\Delta_{\mu })_{L^{p}(s)})=\left\{ f\in
L^{p}(s):\Delta_{\mu }f\in L^{p}(s)\right\},$$
and given by  $(\Delta_{\mu})_{L^{p}(s)}f=\Delta_{\mu}f$. Then, applying the ideas developed in the previous section,  it is enough to study the operator  $S_{\mu }$ in the spaces $L^{p}(sr^{p})$ (or $L^{\infty}(r)$).

\section{Fractional powers of $S_{\mu }$ in Lebesgue spaces}

Now,  we study the non negativity of the part  in  $L^{p}(sr^{p})$ and in $L^{\infty }(r)$ of distributional differential operator $S_{\mu}$
given by (\ref{besselOp2}).

Let $1\leq p<\infty$. We will denote by $S_{\mu,p }$ the part of $S_{\mu }$ in $L^{p}(sr^{p})$;  i.e. the operator $S_{\mu }$ with domain
\begin{equation*}
D\left( S_{\mu,p }\right) =\left\{ f\in
L^{p}(sr^{p}):S_{\mu }f\in L^{p}(sr^{p})\right\}
\end{equation*}
and given by  $ S_{\mu,p }f=S_{\mu }f.$
\vskip.1in

\noindent Analogously, by $S_{\mu,\infty }$ we will denote the part of  $S_{\mu }$ in $L^{\infty }(r)$;  namely, the operator $S_{\mu }$
with domain
\begin{equation*}
D\left( S_{\mu,\infty}\right) =\left\{ f\in L^{\infty
}(r):S_{\mu }f\in L^{\infty }(r)\right\}
\end{equation*}
and $S_{\mu,\infty}f=S_{\mu}f.$

\vskip.2in

In order to study the non negativity of operators $-S_{\mu,\infty}$ and $-S_{\mu,p}$ we consider the following function:

\begin{equation}
\mathcal{K}_{\nu }(x)=\frac{1}{2}\Bigl(\frac{1}{2}x\Bigr)^{\nu }\int_{0}^{\infty
}e^{-t-\frac{x^{2}}{4t}}\frac{dt}{t^{\nu +1}},  \label{rep integ 1}
\end{equation}
for $x\in (0,\infty)$. Since for $\nu < 0$
$$\int_{0}^{\infty
}e^{-t-\frac{x^{2}}{4t}}t^{-\nu -1}\:dt<\int_{0}^{\infty
}e^{-t}t^{-\nu -1}\:dt<\infty$$

\noindent and for $\nu\geq 0$ the function $e^{-t-\frac{x^{2}}{4t}}t^{-\nu -1}$ is bounded in a neighborhood of zero, $\mathcal{K}_{\nu}$ is well defined for $\nu\in \mathbb{R}$ and $\mathcal{K}_{\nu}>0$.

\begin{remark} For non-integer values of $\nu $ , $\mathcal{K}_{\nu}$  (see \cite[(15), pp. 183]{wa}), coincides with the  Macdonald's function  $\rm {K}_{\nu }$  (see \cite[(6) and (7), pp. 78]{wa}) given by
\begin{equation*}
\rm{K}_{\nu }(x)=\frac{\pi }{2}\frac{I_{-\nu }(x)-I_{\nu }(x)}{\sin \nu
\pi }\quad x>0,
\end{equation*}
  with $I_{\upsilon }$ is the modified Bessel function over $(0,\infty )$ (see \cite[(2), pp. 77]{wa}). For integers values of $\nu$, $\rm {K}_{\nu }$ is defined by
\begin{equation*}
\rm{K}_{n}(x)=\underset{\nu \rightarrow n}{\lim }\:\rm{K}_{\nu
}(x)\quad x>0.
\end{equation*}
\end{remark}
\vskip.3in

\noindent Now, given $\lambda >0$,  we consider the function
\begin{equation*}
N_{\lambda }(x)=\lambda ^{\frac{\mu }{2}}x^{\frac{1}{2}}\mathcal{K}_{\mu }(\sqrt{\lambda }\:x), \quad x\in (0,\infty).
\end{equation*}
\vskip.1in

\noindent The following lemmas describe properties  of the kernel $N_{\lambda}$ which are crucial for the study of the no-negativity of Bessel operator (for proofs see Appendix).

\begin{lemma}\label{lemaNlambda}
Given  $\mu >-\frac{1}{2}$ and $\lambda >0$ then
\vskip.1in

a) $N_{\lambda }\in L^{1}(sr)=L^{1}(\frac{r^{-1}}{c_{\mu }})$ and
\begin{equation*}
\left\| N_{\lambda }\right\| _{L^{1}(sr)}=\frac{1}{\lambda }.
\end{equation*}

b) \begin{equation*}
\mathit{h}_{\mu }N_{\lambda }(y)=\frac{y^{\mu +\frac{1}{2}}}{\lambda +y^{2}}.
\end{equation*}

\end{lemma}

\begin{lemma}\label{lemaHankelNlambda}
Let $1\leq p<\infty $. If  $f\in L^{p}(sr^{p})$  or  $L^{\infty }(r)$ then the following equality holds on $\mathcal{H}_{\mu }^{\prime }$
\begin{equation}
\mathit{h}_{\mu }( N_{\lambda }\sharp f) =\frac{1}{\lambda +y^{2}}\:\mathit{h}_{\mu }( f)   \label{hank conv nl}
\end{equation}
\end{lemma}

\vskip.3in
Finally, we can establish the main result of this section.
\begin{teo}\label{teonongLp}
Given $\mu >-\frac{1}{2}$, then

\begin{enumerate}
\item  The operators $ S_{\mu,p}$ and $S_{\mu,\infty}$ are closed .

\item  The operators $-S_{\mu,p}$ and  $-S_{\mu,\infty}$ are non-negative.
\end{enumerate}
\end{teo}

\begin{proof}
(1) Let $\{f_{n}\}_{n=1}^{\infty}\subset D\left( S_{\mu,\infty}\right)$ such that
$$\lim_{n\rightarrow \infty}f_{n}=f\quad \quad {\rm and} \quad\quad \lim_{n\rightarrow
\infty}S_{\mu,\infty} f_{n}=g$$

\noindent in $L^{\infty}(r)$. Since convergence in $L^{\infty}(r)$
implies convergence in $\mathcal{D'}(0,\infty)$, then given $\phi
\in\mathcal{D}(0,\infty)$
$$(S_{\mu}f,\phi)=(f,S_{\mu}\phi)=\lim_{n\rightarrow
\infty}(f_{n},S_{\mu}\phi)=\lim_{n\rightarrow
\infty}(S_{\mu}f_{n},\phi)=(g,\phi),$$

\noindent so, $S_{\mu}f=g$ and $ S_{\mu,\infty}$
is closed. The case of $S_{\mu,p}$ is similar. \\

   \vskip.1in

 (2) Let $\lambda >0$ and $f\in D\left( S_{\mu,\infty}\right)$  such that  $\bigl(\lambda
-S_{\mu,\infty}\bigr)\:f=0$. Then $\left( \lambda -S_{\mu,\infty}\right)f\in L^{\infty }(r)$ and  is
null as regular element of  $\mathcal{H}_{\mu }^{\prime }$, so
\begin{equation*}
\mathit{h}_{\mu }\bigl(\lambda - S_{\mu,\infty}f\bigr)=0
\end{equation*}
in $\mathcal{H}_{\mu }^{\prime }$.  By Proposition
\ref{propHankelSmu}, we obtain that
\begin{equation*}
(\lambda +y^{2})\mathit{h}_{\mu }f=0
\end{equation*}
in $\mathcal{H}_{\mu }^{\prime }$, and hence by Proposition
\ref{operhmu}

\begin{equation*}
\mathit{h}_{\mu }f=(\lambda +y^{2})^{-1}(\lambda
+y^{2})\mathit{h}_{\mu }f=0.
\end{equation*}
Then,  $f=0$ as element of $\mathcal{H}_{\mu }^{\prime }$ and by
Remark \ref{regdistrib}   we conclude that $f=0$ a.e in
$(0,\infty)$ and, $\lambda - S_{\mu,\infty}$ is
inyective. Now, let  $f\in L^{\infty }(r)$ and $g=N_{\lambda }\sharp
\:\mathit{f}$. Then, by Theorem \ref{teoyoung},  $g\in L^{\infty
}(r)$ and
\begin{equation*}
\mathit{h}_{\mu }\bigl(\left(\lambda - S_{\mu,\infty}\right) g\bigr)=(\lambda +y^{2})\:\mathit{h}_{\mu }g=(\lambda
+y^{2})\:\mathit{h}_{\mu }( N_{\lambda }\sharp \:f) =\mathit{h}_{\mu}f.
\end{equation*}
By injectivity of Hankel transform in $\mathcal{H}_{\mu
}^{\prime }$ we obtain that
\begin{equation*}
\left(\lambda - S_{\mu,\infty}\right) g=f,
\end{equation*}
so, $\lambda - S_{\mu,\infty }$ is onto. Also,
\begin{equation*}
\left\| \left(\lambda - S_{\mu,\infty}\right)^{-1}f\right\| _{L^{\infty }(r)}=\left\| g\right\| _{L^{\infty }(r)}=\left\|
N_{\lambda }\sharp \mathit{f}\right\| _{L^{\infty }(r)}\leq \left\|
N_{\lambda }\right\| _{L^{1}(rs)}\left\| \mathit{f}\right\| _{L^{\infty
}(r)}=\frac{1}{\lambda }\left\| \mathit{f}\right\| _{L^{\infty }(r)}
\end{equation*}
hence  $- S_{\mu,\infty }$ is non-negative. The proof of non-negativity of $- S_{\mu,p}$ is similar.
\end{proof}
\begin{remark}: In \cite{motri}, the result of  theorem above has been obtained  for the
particular case $p=2$ and in $\mathbb{R}^{n}_{+}$.
\end{remark}

Now, in view of non-negativity of  $- S_{\mu,\infty }$ and $- S_{\mu,p}$ we can consider the complex fractional powers.
 If $\alpha\in\mathbb{C}$,  $\textrm{Re}\:\alpha> 0$ and $n>\textrm{Re}\:\alpha$ then the fractional power of $- S_{\mu,\infty }$ can be represented by:
$$(- S_{\mu,\infty })^{\alpha}=(- S_{\mu,\infty }+1)^{n}\mathcal{J}_{\infty }^{\alpha}(- S_{\mu,\infty }+1)^{-n},$$
(see \cite[(5.20), pp. 114]{ms}), where with $\mathcal{J}_{\infty}^{\alpha}$ we denote the  Balakrishnan operator associated to $- S_{\mu,\infty}$ given by:
$$\mathcal{J}_{\infty}^{\alpha}\phi=\frac{\Gamma(n)}{\Gamma(\alpha)\Gamma(n-\alpha)}\int_{0}^{\infty}\lambda^{\alpha-1}\Bigl[- S_{\mu,\infty}(\lambda- S_{\mu,\infty})^{-1}\Bigr]^{n}\phi\:d\lambda,$$
for $\alpha$ and $n$ in the previous conditions and $\phi\in D((- S_{\mu,\infty})^{n})$, (see \cite[(3.4), pp. 59]{ms}). The case of  $(- S_{\mu,p})^{\alpha}$ is analogous.

\section{Nonnegativity of Bessel operator $S_{\mu }$ in the space $\mathcal{B}$ }
In order to study non-negativity of Bessel operator in a locally convex space, we begin with the following observation:
 \begin{remark} The continuous operator
 $-S_{\mu}:\mathcal{H}_{\mu}\to\mathcal{H}_{\mu}$ is not  non-negative.
 \end{remark}
 \vskip.1in

Indeed, if we suppose that $-S_{\mu}$ is non-negative in $\mathcal{H}_{\mu}$, by the continuity of $-S_{\mu}$ in $\mathcal{H}_{\mu}$,
  given $\alpha\in \mathbb{C}$, $0<\alpha< 1$ and according to \cite[Chapter 5, pp. 105 and 134]{ms}),  we have that fractional power $(-S_{\mu})^{\alpha}$ would be given by
\begin{equation}\label{e6}
(-S_{\mu})^{\alpha}\phi=\frac{\sin
\alpha\pi}{\pi}\int_{0}^{\infty}\:
 \lambda^{\alpha-1}(-S_{\mu})(\lambda-S_{\mu})^{-1}\phi\:
 d\lambda
 \end{equation}
 \noindent  and $D((-S_{\mu})^{\alpha})=D(-S_{\mu})=\mathcal{H}_{\mu}$. Applying the Hankel transform in
 (\ref{e6})  we obtain
$$
{\it h}_{\mu}\Bigl((-S_{\mu})^{\alpha}\phi\Bigr)(y)=\frac{\sin
\alpha\pi}{\pi}\int_{0}^{\infty}\:
 \lambda^{\alpha-1}y^{2}(\lambda+y^{2})^{-1}{\it
h}_{\mu}\phi(y)\:
 d\lambda.$$
 $$=(y^{2})^{\alpha}{\it h}_{\mu}\phi(y),$$

\noindent  where we have used  equality (3) of Proposition \ref{prophankelsmumas1} and \cite[Remark 3.1.1]{ms}).
In this case it would mean that  $(y^{2})^{\alpha}{\it h}_{\mu}\phi(y)\in\mathcal{H}_{\mu}$ which is false in general (just consider $\phi(y)=y^{\mu +\frac{1}{2}}e^{-y^{2}}$ and $\alpha=\frac{1}{4}$).

\vskip.3in

\noindent Now, we consider  the Banach space
$Y=L^{1}(sr)\cap L^{\infty }(r)$ with the norm
\begin{equation*}
\left\| f\right\| _{Y}=\max \left( \left\| f\right\|
_{L^{1}(sr)},\left\| f\right\| _{L^{\infty }(r)}\right) ,
\end{equation*}
and the part of the Bessel operator in $Y$,
$(S_{\mu})_{Y}$, with domain
$$D\bigl[(S_{\mu})_{Y}\bigr]=\{f\in Y:S_{\mu}f\in Y\}.$$

\noindent From Theorem \ref{teonongLp} it is evident that $-(S_{\mu})_{Y}$
is closed and nonnegative. We have the following proposition:

\begin{proposition}\label{suavdomSmuy} $D\bigl[(S_{\mu})_{Y}\bigr]\subset C_{0}(0,\infty)$.
\end{proposition}
\begin{proof}  By  (\ref{inclusHm}) and (\ref{eqincL1}),  $L^{1}(sr)\cap L^{\infty}(r)\subset L^{1}(0,\infty)\bigcap
L^{2}(0,\infty)$, then for
$f\in D\bigl[(S_{\mu})_{Y}\bigr]$,  $f$ and $S_{\mu}f$ are in
$L^{1}(0,\infty)$. By Remark \ref{obsHankelL1} (see Apendix) then $\mathit{h}_{\mu}f-\mathit{h}_{\mu}S_{\mu}f$ are in
$L^{\infty}(0,\infty)$. By (1) of Proposition \ref{propHankelSmu} we have that
$$(1+y^{2})|\mathit{h}_{\mu}f|\leq M,$$
so, $\mathit{h}_{\mu}f\in L^{1}(0,\infty)$.

We have thus proved that for $f\in D\bigl[(S_{\mu})_{Y}\bigr]$ then
$f$ and $\mathit{h}_{\mu}(f)$ are in $L^{1}(0,\infty)\bigcap
L^{2}(0,\infty)$. Then, by Remark \ref{inverHankelHmu} (see Appendix), we obtain that
$$\mathit{h}_{\mu}(\mathit{h}_{\mu}(f))(x)=f(x) , \quad  \mbox
{a.e}\quad x\in (0,\infty).$$

Since $\mathit{h}_{\mu}(f)\in L^{1}(0,\infty)$ then by Proposition
\ref{ImHankL1}, $f=g$ a.e in $(0,\infty)$ with $g\in
C_{0}(0,\infty)$.
\end{proof}

\vskip.3in

Now, we consider the following space:
$$\mathcal{B}=\{f\in Y: (S_{\mu})^{k}f\in Y \quad\mbox {for}\quad k=0,1,2,\cdots\}=
\bigcap_{k=0}^{\infty}D\bigl[((S_{\mu})_{Y})^{k}\bigr],$$

\noindent with the seminorms
$$\rho_{m}(f)=\max_{0\leq k\leq m}(\| (S_{\mu})^{k}f \|_{Y}), \quad m=0,1,2,\cdots$$

\begin{remark}\label{B debil H} From Proposition \ref{suavdomSmuy} it is evident that
$\mathcal{B}\subset C^{\infty}(0,\infty)\bigcap C_{0}(0,\infty)$. Moreover, is clear from
(\ref{inclusHm}) that $\mathcal{B}\subset L^{p}(sr^{p})$ for all
$1\leq p<\infty$, and considering (1) of proposition \ref{notas smu}
we have that $\mathcal{H}_{\mu }\subset \mathcal{B}$ and the
topology of $\mathcal{H}_{\mu}$ induced by $\mathcal{B}$ is weaker
than the usual topology given in section 2.  Indeed,
from (\ref{acotainf1}) and (\ref{acota L11}) we have that
\begin{equation}\label{eqacotasemi}
\|\phi\|_{Y}\leq C \left[ \gamma _{0,0}^{\mu }(\phi )+\gamma
_{k,0}^{\mu }(\phi )\right] ,\quad \phi \in \mathcal{H}_{\mu}
\end{equation}
for $k> 2\mu+2$, and by continuity of $S_{\mu}$ in $\mathcal{H}_{\mu}$, we deduce
that given a seminorm $\rho_{m}$ there exists a finite set of
seminorms $\{\gamma _{m_{i},k_{i}}^{\mu}\}_{i=1}^{r} $ and constants
$c_{1},\dots, c_{r}$ such that
$$\rho_{m}(\phi)\leq \sum c_{i}\gamma _{m_{i},k_{i}}^{\mu}(\phi),
\quad \phi \in \mathcal{H}_{\mu}.$$
Moreover, $\mathcal{H}_{\mu }$ is dense in $\mathcal{B}$ because $\mathcal{D}(0,\infty)$ it is.
\end{remark}
\begin{proposition}\label{propo nonormable} $\mathcal{B}$ is not normable.
\end{proposition}
\begin{proof} Suppose that $\mathcal{B}$ is locally bounded, then there exists an integer positive $n$ such that the set
$$V_{n}=\Bigl\{\phi\in\mathcal{B} : \rho_{n}(\phi)< \frac{1}{n}\Bigr\},$$
are bounded. Then, there exists a $t_{n}> 0$ such that
\begin{equation}\label{eqinclVn}
V_{n}\subset t_{n}V_{n+1}.
\end{equation}
Let $\phi\in\mathcal{B}$ and $\varphi=\bigl((n+1)\rho_{n}(\phi)\bigr)^{-1}\phi$. Then  $\varphi\in V_{n}$ and by (\ref{eqinclVn}) $(t_{n})^{-1}\varphi\in V_{n+1}$, and hence $\rho_{n+1}((t_{n})^{-1}\varphi)< \frac{1}{n+1}$, so
\begin{equation} \label{desigrho}
\rho_{n+1}(\phi)\leq t_{n}\rho_{n}(\phi)
\end{equation}
Given a constant $l>0$  and $f,g\in C^{2k}(0,\infty)$ related by $f(x)=g(lx)$, we have that

\begin{equation}\label{eqSmuiterated}
(S_{\mu})^{k}f(x)=(l^{2})^{k}((S_{\mu})^{k}g)(lx).
\end{equation}
Now, let $\phi\in\mathcal{B}$ such that $(S_{\mu})^{n+1}\phi$ non-identically vanishing function and a constant  $s>1$. Then $\psi(x)=\phi(s^{-1}x)$ remains in $\mathcal{B}$ and verified that $(S_{\mu})^{n+1}\psi$ is a non-identically vanishing function and  $\phi(x)=\psi(sx)$. Then,
$$
\|(S_{\mu})^{n+1}\psi\|_{L^{\infty}(r)}=s^{-\mu-\frac{1}{2}}s^{-2(n+1)}\|(S_{\mu})^{n+1}\phi\|_{L^{\infty}(r)}\leq$$
 \begin{equation}\label{eqnonorm}
 s^{-\mu-\frac{1}{2}}s^{-2(n+1)} t_{n}\rho_{n}(\phi)\leq s^{-\mu-\frac{1}{2}}s^{-2(n+1)}s^{\mu+\frac{1}{2}} s^{2n} t_{n}\rho_{n}(\psi)=s^{-2}t_{n}\rho_{n}(\psi).
\end{equation}
Since (\ref{eqnonorm}) is verified for all $s>1$, taking $s\rightarrow \infty$,  we conclude that $\|(S_{\mu})^{n+1}\psi\|_{L^{\infty}(r)}$=0 which contradicts the assumption about $\psi$. Then the proposition follows.
\end{proof}

We denote with $(S_{\mu})_{\mathcal{B}}$ the part of Bessel operator $S_{\mu}$ in $\mathcal{B}$. By definition of $\mathcal{B}$, it is evident that the domain of  $(S_{\mu})_{\mathcal{B}}$ is $\mathcal{B}$ and is verified the following result
\begin{teo}\label{cont smu in B} $\mathcal{B}$ is a Fr\'echet space and $-(S_{\mu})_{\mathcal{B}}$ is continuous and nonnegative operator
on $\mathcal{B}$.
\end{teo}
\begin{proof} The proof is immediate by Proposition 1.4.2  given in \cite{ms}.
\end{proof}

\section{Nonnegativity of Bessel operator $S_{\mu }$ in the distributional space  $\mathcal{B'}$ }
In this section we study the non-negativity of Bessel operator in the topological dual space of $\mathcal{B}$ with the strong topology, i.e the space $\mathcal{B'}$ with the seminorms $\{|.|_{B}\}$, where the sets  $B$ are in the family of bounded sets in $\mathcal{B}$,  and are given by

$$|T|_{B}=\sup_{\phi\in B}|(T,\phi)|, \quad  T\in \mathcal{B'}.$$

\begin{remark} As in \cite[Remark 3.4, pp.263]{stu}, $\mathcal{B'}$ is sequentially complete because $\mathcal{B}$ is non-normable. Moreover, and for $1\leq p\leq \infty$ then $L^{p}(sr^{p})\subset \mathcal{B'}$. To prove this, we observe that given $f\in L^{p}(sr^{p})$ and $\phi\in \mathcal{B}$ and $q$ the conjugate of $p$ then
\begin{equation}\label{eqacotaf}
\Bigl|\int_{0}^{\infty}f\phi\Bigr|=\Bigl|\int_{0}^{\infty}f\phi s^{-1}r^{-p}sr^{p}\Bigr|\leq \|f\|_{L^{p}(sr^{p})}\|\phi s^{-1}r^{-p}\|_{L^{q}(sr^{p})},
\end{equation}

\noindent and
$$
\|\phi s^{-1}r^{-p}\|_{L^{q}(sr^{p})}=\Bigl\{\int_{0}^{\infty}|\phi s^{-1}r^{-p}|^{q}sr^{p}\Bigr\}^{\frac{1}{q}}=
\Bigl\{\int_{0}^{\infty}|\phi|^{q}(c_{\mu}r^{2}r^{-p})^{q}sr^{p}\Bigr\}^{\frac{1}{q}}=$$
\begin{equation}\label{eq phi}
=c_{\mu}\Bigl\{\int_{0}^{\infty}|\phi|^{q}r^{2q-pq+p} s \Bigr\}^{\frac{1}{q}}
=c_{\mu}\Bigl\{\int_{0}^{\infty}|\phi|^{q} sr^{q}\Bigr\}^{\frac{1}{q}}.
\end{equation}

\noindent Moreover, by (\ref{acota Lp}) we have that
\begin{equation}\label{acota phi}
\|\phi\|_{L^{q}(sr^{q})}\leq \rho_{0}(\phi),
\end{equation}
and from (\ref{eqacotaf}), (\ref{eq phi}) and (\ref{acota phi}) we obtain that $f\in\mathcal{B'}$.

Now, let $B$ be a bounded set in $\mathcal{B}$ then
$$\sup_{\phi\in \mathcal{B}}\Bigl|\int_{0}^{\infty}f\phi\Bigr|\leq c_{\mu}\|f\|_{L^{p}(sr^{p})}\sup_{\phi\in \mathcal{B}}\|\phi\|_{L^{q}(sr^{q})}\leq c_{\mu}\|f\|_{L^{p}(sr^{p})}\sup_{\phi\in \mathcal{B}}\rho_{0}(\phi).$$
Consequently, the topology in $L^{p}(sr^{p})$ induced by $\mathcal{B'}$ with strong topology is weaker than the usual topology.
\end{remark}
\begin{remark} By Remark \ref{B debil H},  $\mathcal{B'}\subset \mathcal{H'}_{\mu}$. Moreover, from the continuity of  the Bessel operator in $\mathcal{B}$, we can consider $S_{\mu}$  in $\mathcal{B'}$ as adjoint operator of $S_{\mu}$  in $\mathcal{B}$, that is
$$(S_{\mu}T,\phi)=(T,S_{\mu}\phi), \quad T\in\mathcal{B'}, \phi\in\mathcal{B},$$
and we denote with $(S_{\mu})_{\mathcal{B'}}$ the part of Bessel operator in $\mathcal{B'}$
\end{remark}
\begin{teo} The operator $-(S_{\mu})_{\mathcal{B'}}$ is continuous and non negative considering the strong topology in $\mathcal{B'}$.
\end{teo}
\begin{proof} The proof of continuity is identical to the proof given in \cite[Theorem 3.5, pp. 264]{stu} for the Laplacean operator and the non negativity is a consecuence of theory of fractional powers in distributional spaces (see \cite[pp. 24]{ms}).
\end{proof}
\begin{remark} The operator $(S_{\mu})_{\mathcal{B'}}$ is not injective because the function $x^{\mu+\frac{1}{2}}$ is solution of $S_{\mu}=0$ and belongs to $\mathcal{B'}$, in fact
$$|(x^{\mu+\frac{1}{2}},\phi)|\leq c_{\mu}\|\phi\|_{L^{1}(sr)}\leq c_{\mu}\rho_{0}(\phi),  \:\: (\phi\in\mathcal{B}).$$
\end{remark}
\vskip.2in
According to representation of fractional powers of operators in locally convex spaces given in \cite{ms}, for $\textrm{Re}\:\alpha>0$, $n>\textrm{Re}\:\alpha$ , $T\in\mathcal{B'}$, $(-(S_{\mu})_{\mathcal{B'}})^{\alpha}$ is given by

$$(-(S_{\mu})_{\mathcal{B'}})^{\alpha}T=\frac{\Gamma(n)}{\Gamma(\alpha)\Gamma(n-\alpha)}\int_{0}^{\infty}\lambda^{\alpha-1}\Bigl[-(S_{\mu})_{\mathcal{B'}}(\lambda- (S_{\mu})_{\mathcal{B'}})^{-1}\Bigr]^{n}T\:d\lambda.$$

From the general theory of fractional power in sequentially complete locally convex spaces (see \cite[ pp. 134]{ms}), we deduce immediately the properties of powers as multiplicativity, spectral mapping theorem, and
\vskip.2in

1) If $\textrm{Re}\:\alpha> 0$ then

\begin{equation}\label{eqduality}
\Bigl((-(S_{\mu})_{\mathcal{B}})^{\alpha}\Bigr)^{\ast}=\Bigl((-(S_{\mu})_{\mathcal{B}})^{\ast}\Bigr)^{\alpha}.
\end{equation}
Since $(-(S_{\mu})_{\mathcal{B}})^{\ast}=-(S_{\mu})_{\mathcal{B'}}$ then from (\ref{eqduality}) we obtain the following duality formula
$$((-(S_{\mu})_{\mathcal{B'}})^{\alpha}T,\phi)=(T,(-(S_{\mu})_{\mathcal{B}})^{\alpha}\phi), \:\: (\phi\in\mathcal{B}, T\in\mathcal{B'}).$$
\vskip.2in

2) Since the usual topology in $L^{p}(sr^{p})$ is stronger than the topology induced  by $\mathcal{B'}$ then we can deduce that

$$\bigl[(-(S_{\mu})_{\mathcal{B'}})^{\alpha}\bigr]_{L^{p}(sr^{p})}=((-(S_{\mu,p}))^{\alpha},$$
if $\textrm{Re}\:\alpha> 0$, (see \cite[Theorem 12.1.6,  pp. 284]{ms}).

\section{Appendix}
\subsection{Some properties of Hankel transform in Lebesgue spaces}
\begin{proposition}\label{propplan} If  $f,g\in L^{1}(sr)$, then:

(1)  $\mathit{h}_{\mu }f\in L^{\infty }(r).$

(2)
\begin{equation} \label{parseval}
\int_{0}^{\infty }\mathit{h}_{\mu }f g=\int_{0}^{\infty}f\mathit{h}_{\mu}g.
\end{equation}

\end{proposition}

\begin{proof}
The assertion  $(1)$ follows from the equalities
\begin{equation*}
x^{-\mu -\frac{1}{2}}\mathit{h}_{\mu }(f)(x)=x^{-\mu -\frac{1}{2}%
}\int_{0}^{\infty }\sqrt{xy}J_{\mu }(xy)f(y)dy=\int_{0}^{\infty }(xy)^{-\mu
}J_{\mu }(xy)f(y)y^{\mu +\frac{1}{2}}dy.
\end{equation*}

\noindent The existence and equality of the integrals in (\ref{parseval}) are an immediate consequence of Tonelli-Hobson
theorem.
\end{proof}

In \cite{HI} is studied a version of Hankel transform given by:

\begin{equation*}
H_{\mu }(f)(x)=c_{\mu }\int_{0}^{\infty }\left( xy\right) ^{-\mu
}J_{\mu }(xy)f(y)s(y)dy,
\end{equation*}
for $f\in L^{1}(s)$. $H_{\mu}$ is related whit $\mathit{h}_{\mu}$ by
\begin{equation*}
\mathit{h}_{\mu }(f)=r^{-1}H_{\mu }(rf).
\end{equation*}
for $f\in L^{1}(sr)$ ($r$ and $s$ like as in section 2). From this relation and the inversion theorem
for  $H_{\mu}$ (see \cite[Corollary 2e, pp 316]{HI}) , we obtain the
following inversion  theorem for $\mathit{h}_{\mu}$

\begin{proposition}\label{invHankel} If  $f\in L^{1}(sr)$ and $\mathit{h}_{\mu }(f)\in L^{1}(sr)$  then $f$ may be
redefined on a set of measure zero so that it is continuous in
$(0,\infty)$ and

\begin{equation}
f(x)=\int_{0}^{\infty }\sqrt{xy}J_{\mu
}(xy)\mathit{h}_{\mu}(f)(y)dy=\mathit{h}_{\mu}(\mathit{h}_{\mu}(f))(x)
\end{equation}

\end{proposition}

\begin{remark} \label{inverHankelHmu}
From the above Propositon we deduce immediately the validity of equality $\mathit{h}_{\mu}\mathit{h}_{\mu}f=f$ in $\mathcal{H}_{\mu}$ and $\mathcal{H'}_{\mu}$.

\end{remark}

\vskip.3in

 With $L^{p}(0,\infty)$ we denote the usual Lebesgue space
given by (\ref{defespLp}) with   $w(x)=1$.

\begin{remark}\label{obsHankelL1} Since the function $(z)^{\frac{1}{2}}J_{\mu}(z)$
is bounded in $(0,\infty)$ for $\mu > -\frac{1}{2}$, then for $f\in
L^{1}(0,\infty)$ we have that $\mathit{h}_{\mu}f$ is continuous
and $\|\mathit{h}_{\mu}f\|_{\infty}\leq C\|f\|_{1}$.
\end{remark}
As usual, we denote with $C_{0}(0,\infty)$ the set of continuous
functions in $(0,\infty)$ and vanishes at infinity. We have the
following proposition:

\begin{proposition} \label{ImHankL1} $\mathit{h}_{\mu}(L^{1}(0,\infty))\subset C_{0}(0,\infty)$

\end{proposition}
\begin{proof}
First, we observe that
\begin{equation}\label{eqincL1}
 L^{1}(sr)\cap L^{\infty }(r)\subset L^{1}(0,\infty).
 \end{equation}

\noindent Indeed,
 $$\int_{0}^{\infty}|f|\:dx=\int_{0}^{\infty}|f| rr^{-1}\:dx=\int_{0}^{1}|f| rr^{-1}\:dx+\int_{1}^{\infty}|f|
 rr^{-1}\:dx\leq $$
 $$\leq \| f \|_{L^{\infty}(r)}\int_{0}^{1}r^{-1}\:dx+\int_{1}^{\infty}|f|
r^{-1}\:dx =C \| f\|_{L^{\infty}(r)}+c_{\mu }\| f\|_{L^{1}(rs)},$$ because
$r< 1$ in $[1,\infty)$ , $\mu +\frac{1}{2}> 0$ and $rs=c^{-1}_{\mu }r^{-1}$.

By (\ref{eqincL1}) and (\ref{inclusHm}) we deduce that $\mathcal{H}_{\mu}\subset L^{1}(0,\infty)$. Since $\mathcal{D}(0,\infty)\subset \mathcal{H}_{\mu}$  then
$\mathcal{H}_{\mu}$ is dense in $L^{1}(0,\infty)$. Given $f\in L^{1}(0,\infty)$ and
$\{\phi_{n}\}\subset \mathcal{H}_{\mu}$ such that $\phi_{n}\rightarrow f$ in $L^{1}(0,\infty)$ then by Remark \ref{obsHankelL1} (see Apendix)
$\mathit{h}_{\mu}(\phi_{n})\rightarrow \mathit{h}_{\mu}(f)$ uniformly. Since
 $\mathit{h}_{\mu}(\phi_{n})\in C_{0}(0,\infty)$ then $\mathit{h}_{\mu}(f)\in C_{0}(0,\infty)$.

\end{proof}
\begin{remark} For $\mu > -\frac{1}{2}$,  $\mathcal{H}_{\mu}$ is a dense subset of $L^{2}(0,\infty)$ and  for $\phi\in \mathcal{H}_{\mu}$ we have that

$$\|\mathit{h}_{\mu}\phi \|_{2}=\|\phi\|_{2},$$
So, we can consider the extension to $L^{2}(0,\infty)$ of $\mathit{h}_{\mu}$ and
$$\|\mathit{h}_{\mu}f\|_{2}=\|f\|_{2},$$
for $f\in L^{2}(0,\infty)$.

\end{remark}
\subsection{Hankel convolution}
\begin{proposition}\label{propD} $D_{\mu }(x,y,z)$ satisfies the following properties:
\begin{enumerate}
\item $D_{\mu }(x,y,z)\geq 0$  for  $x,y,z \in (0,\infty)$.
\item $\int_{0}^{\infty} \sqrt{zt} J_{\mu}(zt) D_{\mu }(x,y,z)\: dz=\sqrt{xt} J_{\mu}(xt)\sqrt{yt} J_{\mu}(yt)t^{-\mu-\frac{1}{2}}$ for  $x,y,t\in (0,\infty)$.
\item $\int_{0}^{\infty} z^{\mu+\frac{1}{2}}D_{\mu }(x,y,z)\: dz=c_{\mu}^{-1}x^{\mu+\frac{1}{2}}y^{\mu+\frac{1}{2}}$  for  $x,y \in (0,\infty)$.
\end{enumerate}
\end{proposition}
\begin{proof}  Assertion  (1) follows immediately.

\noindent  To proof  (2), first we observe that
$$\Bigl|\sqrt{zt} J_{\mu}(zt) D_{\mu }(x,y,z)\Bigr|=\Bigl|\sqrt{zt}(zt)^{\mu} (zt)^{-\mu}J_{\mu}(zt) D_{\mu }(x,y,z)\Bigr|\leq C t^{\mu+\frac{1}{2}}\Bigl|z^{\mu+\frac{1}{2}} D_{\mu }(x,y,z)\Bigr|=$$
$$=C t^{\mu+\frac{1}{2}}\Biggl|z^{\mu+\frac{1}{2}}\frac{2^{\mu-1}(xyz)^{-\mu+\frac{1}{2}}}{\Gamma(\mu+\frac{1}{2})\sqrt{\pi}}(A(x,y,z))^{2\mu-1}\Biggr|=$$
$$=C t^{\mu+\frac{1}{2}}\frac{(xy)^{-\mu+\frac{1}{2}}}{2^{3\mu-1}\Gamma(\mu+\frac{1}{2})\sqrt{\pi}}\Biggl|z((x+y)^{2}-z^{2})^{\mu-\frac{1}{2}}(z^{2}-(x-y)^{2})^{\mu-\frac{1}{2}}
\Biggr|$$
and the last function is integrable for  $z\in [\:\vert x-y \vert\:,\:x+y\:]$ and $\mu>-\frac{1}{2}$. So, we conclude that $z^{\mu+\frac{1}{2}}D_{\mu }(x,y,z)$ is integrable in $(0,\infty)$ for $\mu>-\frac{1}{2}$. Now, we consider the change of variables  $T:(0,\pi)\rightarrow (0,\infty)$ given by $T(\phi)=\sqrt{x^{2}+y^{2}-2xy \cos \phi}$.
Then $\vert x-y \vert< T(\phi)<x+y$ , $\frac{d}{d\phi}T(\phi)=\frac{xy \sin \phi}{\sqrt{x^{2}+y^{2}-2xy \cos \phi}}$ and $A(x,y,T(\phi))=\frac{xy}{2}\sin \phi$. So,
$$\int_{0}^{\infty} \sqrt{zt} J_{\mu}(zt) D_{\mu }(x,y,z)\: dz=\frac{2^{\mu-1}(xy)^{-\mu+\frac{1}{2}}}{\Gamma(\mu+\frac{1}{2})\sqrt{\pi}}\int_{\vert x-y \vert}^{x+y} \sqrt{zt} J_{\mu}(zt)z^{-\mu+\frac{1}{2}}\:(A(x,y,z))^{2\mu-1}\: dz=$$

$$=\frac{2^{\mu-1}(xy)^{-\mu+\frac{1}{2}}t^{\frac{1}{2}}}{\Gamma(\mu+\frac{1}{2})\sqrt{\pi}}\int_{0}^{\pi} J_{\mu}\bigl(\sqrt{x^{2}+y^{2}-2xy \cos \phi}\:\:t\bigr) \bigl(\sqrt{x^{2}+y^{2}-2xy \cos \phi}\bigr)^{-\mu+1}.$$
$$.\Bigl(\frac{xy}{2}\sin \phi\Bigr)^{2\mu-1}\frac{xy \sin \phi}{\sqrt{x^{2}+y^{2}-2xy \cos \phi}}\:\:d\phi=$$
\begin{equation}\label{eqint4}
\frac{(xy)^{\mu+\frac{1}{2}}t^{\frac{1}{2}}}{2^{\mu}\Gamma(\mu+\frac{1}{2})\sqrt{\pi}}\int_{0}^{\pi}\frac{ J_{\mu}\bigl(\sqrt{x^{2}+y^{2}-2xy \cos \phi}\:\:t\bigr)}{ \bigl(\sqrt{x^{2}+y^{2}-2xy \cos \phi}\bigr)^{\mu}}\sin^{2\mu}\phi\:\:d\phi.
\end{equation}

\noindent Since
$$\int_{0}^{\pi} \frac{J_{\mu}\bigl(\sqrt{Z^{2}+z^{2}-2zZ \cos \phi}\bigr)}{\bigl(\sqrt{Z^{2}+z^{2}-2zZ \cos \phi}\bigr)^{\mu}}\sin^{2\mu}\phi\:\:d\phi=2^{\mu}\Gamma\Bigl(\mu+\frac{1}{2}\Bigr)\Gamma\Bigl(\frac{1}{2}\Bigr)\frac{J_{\mu}(Z)}{Z^{\mu}}
\frac{J_{\mu}(z)}{z^{\mu}},$$
(see (16) pg. 367 \cite{wa}) valid to $z,Z>0$, $\mu> -\frac{1}{2}$, and considering $Z=xt$ y $z=yt$ we obtain in the last equality
\begin{equation}\label{eqint5}
\int_{0}^{\pi} \frac{J_{\mu}\bigl(\sqrt{x^{2}+y^{2}-2xy \cos \phi}\:\:t\bigr)}{\bigl(\sqrt{x^{2}+y^{2}-2xy \cos \phi}\bigr)^{\mu}t^{\mu}}\sin^{2\mu}\phi\:\:d\phi=2^{\mu}\Gamma\Bigl(\mu+\frac{1}{2}\Bigr)\Gamma\Bigl(\frac{1}{2}\Bigr)\frac{J_{\mu}(xt)}{(xt)^{\mu}}
\frac{J_{\mu}(yt)}{(yt)^{\mu}}.
\end{equation}
Applying (\ref{eqint5}) in (\ref{eqint4}), we obtain that
$$\int_{0}^{\infty} \sqrt{zt} J_{\mu}(zt) D_{\mu }(x,y,z)\: dz=$$

$$\frac{(xy)^{\mu+\frac{1}{2}}t^{\frac{1}{2}}}{2^{\mu}\Gamma(\mu+\frac{1}{2})\sqrt{\pi}} \:\: t^{\mu}2^{\mu}\Gamma\Bigl(\mu+\frac{1}{2}\Bigr)\Gamma\Bigl(\frac{1}{2}\Bigr)\frac{J_{\mu}(xt)}{(xt)^{\mu}}
\frac{J_{\mu}(yt)}{(yt)^{\mu}}=\sqrt{xt} J_{\mu}(xt)\sqrt{yt} J_{\mu}(yt)t^{-\mu-\frac{1}{2}}.$$
\noindent As for (3), we consider again the change of variable $T$
$$
\int_{0}^{\infty} z^{\mu+\frac{1}{2}}D_{\mu }(x,y,z)\: dz=\frac{2^{\mu-1}(xy)^{-\mu+\frac{1}{2}}}{\Gamma(\mu+\frac{1}{2})\sqrt{\pi}}\int_{\vert x-y \vert}^{x+y} z\:(A(x,y,z))^{2\mu-1}\: dz=$$
$$\frac{2^{\mu-1}(xy)^{-\mu+\frac{1}{2}}}{\Gamma(\mu+\frac{1}{2})\sqrt{\pi}}\int_{0}^{\pi}\sqrt{x^{2}+y^{2}-2xy \cos \phi}\:\:\Bigl(\frac{xy}{2}\sin \phi\Bigr)^{2\mu-1} \frac{xy \sin \phi}{\sqrt{x^{2}+y^{2}-2xy \cos \phi}}\:d\phi=$$

$$=\frac{2^{\mu-1}(xy)^{-\mu+\frac{1}{2}}}{\Gamma(\mu+\frac{1}{2})\sqrt{\pi}}\:\:2^{-2\mu+1}(xy)^{2\mu}\:
\frac{\sqrt{\pi}\:\Gamma(\mu+\frac{1}{2})}{\mu\:\Gamma(\mu)}=$$
$$=(xy)^{\mu+\frac{1}{2}} (2^{\mu}\:\Gamma(\mu+1))^{-1}.$$

\end{proof}
\vskip.2in
\noindent Proof of Theorem \ref{teoyoung}
\begin{proof} (1)
Let  $f\in L^{1}(sr)$ and $g\in L^{\infty }(r)$, then:
$$
\int_{0}^{\infty}|f(y)|\biggl[\int_{0}^{\infty}|g(z)|D_{\mu }(x,y,z)dz\biggr]dy \leq $$
$$\leq \left\| g\right\| _{L^{\infty }(r)}\int_{0}^{\infty}|f(y)|\biggl[\int_{0}^{\infty}z^{\mu+\frac{1}{2}} D_{\mu }(x,y,z) dz\biggr]dy =$$
\begin{equation}\label{eqlplinfty}
\left\| g\right\| _{L^{\infty }(r)}\int_{0}^{\infty}|f(y)|y^{\mu+\frac{1}{2}} x^{\mu+\frac{1}{2}}c_{\mu }^{-1} dy=x^{\mu+\frac{1}{2}}\left\| f \right\|
_{L^{1}(sr)}\left\| g\right\| _{L^{\infty }(r)},
\end{equation}
thus ( \ref{conv zem}) exists for every $x\in (0,\infty)$ and by (\ref{eqlplinfty}) we have
$$|x^{-\mu-\frac{1}{2}}(f\sharp g)|\leq \left\| f \right\|
_{L^{1}(sr)}\left\| g\right\| _{L^{\infty }(r)},$$
hence (\ref{young Linfty}).\\
\noindent (2) Given $f\in L^{1}(sr)$  and $g\in L^{p}(sr^{p})$ with $1\leq p <\infty$, set:
\begin{equation}
K(x,z)=\int_{0}^{\infty} x^{-\mu-\frac{1}{2}} z^{-\mu-\frac{1}{2}}f(y)D_{\mu}(x,y,z)c_{\mu } dy.
\end{equation}
We claim that:
$$(1)'\int_{0}^{\infty}|K(x,z)|s(x)dx\leq \left\| f \right\|
_{L^{1}(sr)}=\left\| rf \right\|
_{L^{1}(s)};$$
$$(2)'\int_{0}^{\infty}|K(x,z)|s(z)dz\leq \left\| f \right\|
_{L^{1}(sr)}=\left\| rf \right\|
_{L^{1}(s)}.$$
In fact
$$\int_{0}^{\infty}|K(x,z)|s(x)dx=\int_{0}^{\infty}|K(x,z)|x^{2\mu+1}c_{\mu}^{-1}dx=$$
$$\int_{0}^{\infty}\biggl |\int_{0}^{\infty} x^{-\mu-\frac{1}{2}} z^{-\mu-\frac{1}{2}}f(y)D_{\mu}(x,y,z)c_{\mu } dy \biggr |x^{2\mu+1}c_{\mu}^{-1}dx\leq$$
$$\int_{0}^{\infty}\biggl [\int_{0}^{\infty} x^{\mu+\frac{1}{2}}D_{\mu}(x,y,z)dx \biggr ] |f(y)|z^{-\mu-\frac{1}{2}}dy =$$
$$\int_{0}^{\infty}c_{\mu}^{-1} y^{\mu+\frac{1}{2}}z^{\mu+\frac{1}{2}}|f(y)|z^{-\mu-\frac{1}{2}}dy =\left\| f \right\|
_{L^{1}(sr)}$$
The proof for (2)' is similar.
If  $h\in L^{p}(s)$ then the integral
$$Th(x)=\int_{0}^{\infty}K(x,z)h(z)s(z)dz.  $$
converges absolutely for a.e. $x\in (0,\infty)$ (see \cite{gf}, theorem 6.18), also $Th\in L^{p}(s)$ and
\begin{equation} \label{eqdes1}
\left\| Th\right\|_{L^{p}(s)} \leq \left\| f \right\|
_{L^{1}(rs)} \left\| h \right\|
_{L^{p}(s)}.
\end{equation}
 Then, since $g\in L^{p}(sr^{p})$ then $h=rg\in L^{p}(s)$ and we have that
$$T(rg)(x)=\int_{0}^{\infty}K(x,z)z^{-\mu-\frac{1}{2}}g(z)c_{\mu}^{-1}z^{2\mu+1}dz=$$
$$\int_{0}^{\infty}\Biggl[\int_{0}^{\infty} x^{-\mu-\frac{1}{2}} z^{-\mu-\frac{1}{2}}f(y)D_{\mu}(x,y,z)c_{\mu } dy \Biggr] z^{-\mu-\frac{1}{2}}g(z)c_{\mu}^{-1}z^{2\mu+1}dz=
$$
\begin{equation}\label{eqK}
x^{-\mu-\frac{1}{2}}\int_{0}^{\infty} \Biggl[\int_{0}^{\infty}f(y)D_{\mu}(x,y,z)dy \Biggr] g(z)dz
\end{equation}
With similar considerations applied to $\mid f\mid \in L^{1}(sr)$ and $\mid g\mid \in L^{p}(sr^{p})$ we obtain that for a.e. $x\in (0,\infty)$ the integral
$$\int_{0}^{\infty} \Biggl[\int_{0}^{\infty}\mid f(y)\mid D_{\mu}(x,y,z)dy \Biggr] \mid g(z)\mid \:dz$$
is finite and by application of Tonelli-Hobson theorem in (\ref{eqK}) we conclude that
$$T(rg)(x)=x^{-\mu-\frac{1}{2}} (f\sharp g)(x).$$
From the previous equality and (\ref{eqdes1}) we have
$$
\left\| r (f\sharp g)\right\|_{L^{p}(s)} \leq \left\| f \right\|
_{L^{1}(rs)} \left\| rg \right\|
_{L^{p}(s)},
$$
and so (\ref{young lp}) is valid.
\end{proof}

\vskip.2in
\noindent Proof of Theorem 2

\begin{proof} By (1) and  (3) of Proposition \ref{propD} we have that
$$\int_{0}^{\infty}\int_{0}^{\infty}x_{0}^{-\mu-\frac{1}{2}}y^{\mu+\frac{1}{2}}D_{\mu }(x_{0},y,z) \phi_{n}(z)\:dy dz=1,$$
then
 $$f\sharp \phi_{n}(x_{0})-f(x_{0})=$$
 $$\int_{0}^{\infty}\int_{0}^{\infty}D_{\mu }(x_{0},y,z) \phi_{n}(z)y^{\mu+\frac{1}{2}}\bigl(y^{-\mu-\frac{1}{2}}f(y)-x_{0}^{-\mu-\frac{1}{2}}f(x_{0})\bigr)\:dydz.$$

\noindent  By continuity of  $f$ in $x_{0}$ let $\delta> 0$ such that  $|y^{-\mu-\frac{1}{2}}f(y)-x_{0}^{-\mu-\frac{1}{2}}f(x_{0})|<\varepsilon$  if $|y-x_{0}|<\delta$, and we consider
 $$|f\sharp \phi_{n}(x_{0})-f(x_{0})|\leq |I_{1}|+|I_{2}|,$$
where
\begin{equation}\label{eqI1}
|I_{1}|=\Biggl|\int_{0}^{\delta}\int_{0}^{\infty}D_{\mu }(x_{0},y,z) \phi_{n}(z)y^{\mu+\frac{1}{2}}\bigl(y^{-\mu-\frac{1}{2}}f(y)-x_{0}^{-\mu-\frac{1}{2}}f(x_{0})\bigr)\:dydz \Biggr|
\end{equation}
\begin{equation}\label{eqI2}
|I_{2}|=\Biggl|\int_{\delta}^{\infty}\int_{0}^{\infty}D_{\mu }(x_{0},y,z) \phi_{n}(z)y^{\mu+\frac{1}{2}}\bigl(y^{-\mu-\frac{1}{2}}f(y)-x_{0}^{-\mu-\frac{1}{2}}f(x_{0})\bigr)\:dydz \Biggr|
\end{equation}
Since $D_{\mu }(x_{0},y,z)\neq 0$  only if $|x_{0}-z|<y<x_{0}+z$, and if $0< z<\delta$ then $(|x_{0}-z|,x_{0}+z)\subset (x_{0}-\delta, x_{0}+\delta)$ , then we obtain in (\ref{eqI1}) that
$$|I_{1}|\leq  \varepsilon \int_{0}^{\delta}\int_{0}^{\infty}D_{\mu }(x_{0},y,z) \phi_{n}(z)y^{\mu+\frac{1}{2}}\:dydz= $$
$$\varepsilon \int_{0}^{\delta}x_{0}^{\mu+\frac{1}{2}} z^{\mu+\frac{1}{2}}c_{\mu}^{-1}\phi_{n}(z)\:dz\leq \varepsilon x_{0}^{\mu+\frac{1}{2}}.$$
On the other hand
$$|I_{2}|\leq  2\| f\|_{L^{\infty}(r)}\int_{\delta}^{\infty}\int_{0}^{\infty}D_{\mu }(x_{0},y,z) \phi_{n}(z)y^{\mu+\frac{1}{2}}\:dydz =$$
$$2\| f\|_{L^{\infty}(r)}x_{0}^{\mu+\frac{1}{2}}\int_{\delta}^{\infty}c_{\mu}^{-1}z^{\mu+\frac{1}{2}} \phi_{n}(z)\:dz$$
so, $|I_{2}|\rightarrow 0$  when $n\rightarrow \infty$ and the first assertion has been proven.
Second affirmation follows from the previous proof and the uniformly continuity of $rf$.
\end{proof}
\vskip.2in

\noindent Proof of Proposition 5
\begin{proof} If $f,g\in L^{1}(sr)$, then
$$\mathit{h}_{\mu }\left( f\sharp g\right)(t)=\int_{0}^{\infty} f\sharp g(x)\sqrt{xt}J_{\mu}(xt)\: dx=$$
\begin{equation}\label{eqhakconv}
\int_{0}^{\infty} \Biggl[\int_{0}^{\infty }\int_{0}^{\infty
}D_{\mu }(x,y,z)f(y)g(z)\:dydz\Biggr]\sqrt{xt}J_{\mu}(xt)\: dx
\end{equation}
Since  $f\sharp g\in L^{1}(sr)$ we obtain that
$$\int_{0}^{\infty} \Biggl[\int_{0}^{\infty }\int_{0}^{\infty
}|D_{\mu }(x,y,z)f(y)g(z)|\:dydz\Biggr]|\sqrt{xt}J_{\mu}(xt)|\: dx<\infty$$
for all $t\in (0,\infty)$. Then,
$$\mathit{h}_{\mu }\left( f\sharp g\right)(t)=\int_{0}^{\infty }\int_{0}^{\infty}f(y)g(z)\Biggl[\int_{0}^{\infty} \sqrt{xt}J_{\mu}(xt) D_{\mu }(x,y,z)\: dx\Biggr]\:dydz=
$$
$$\int_{0}^{\infty }\int_{0}^{\infty}f(y)g(z)\sqrt{yt}J_{\mu}(yt) \sqrt{zt}J_{\mu}(zt) t^{-\mu-\frac{1}{2}}\:dydz=t^{-\mu-\frac{1}{2}}
\mathit{h}_{\mu }(f)(t)\mathit{h}_{\mu }(g)(t). $$

\end{proof}

\subsection{Properties of $N_{\lambda}$}  Proof of Lemma \ref{lemaNlambda}

\begin{proof}
a)
$$
\left\| N_{\lambda }\right\| _{L^{1}(sr)} =\frac{1}{c_{\mu }}\int_{0}^{\infty }\lambda ^{\frac{\mu }{2}}x^{\frac{1}{2}}\mathcal{K}_{\mu }
(\sqrt{\lambda}\:x)x^{\mu +\frac{1}{2}}dx=$$
$$=\frac{1}{c_{\mu }}\int_{0}^{\infty} \lambda ^{\frac{\mu}{2}}x^{\frac{1}{2}}\frac{1}{2}\Bigl(\frac{1}{2}\sqrt{\lambda }\:x\Bigr)^{\mu}\left[ \int_{0}^{\infty }e^{-t-\frac{\lambda x^{2}}{4t}}\frac{dt}{t^{\mu +1}}\right] x^{\mu +\frac{1}{2}}dx=$$
$$=\frac{1}{c_{\mu }}\Bigl(\frac{1}{2}\Bigr)^{\mu +1}\lambda ^{\mu }\int_{0}^{\infty
}\left[ \int_{0}^{\infty }x^{2\mu +1}e^{-\frac{\lambda x^{2}}{4t}}dx\right]
e^{-t}\frac{dt}{t^{\mu +1}}=$$
$$=\frac{1}{c_{\mu }}2^{\mu }\lambda ^{-1}\Gamma
(\mu +1)\int_{0}^{\infty }e^{-t}dt=\lambda ^{-1}.$$

\noindent For b), in virtue of the following equality

\begin{equation}\label{eqautov2}
\int_{0}^{\infty} x^{\mu+\frac{1}{2}}e^{-\frac{x^{2}}{2}}\sqrt{xy}\:J_{\mu}(xy)\:dx=y^{\mu+\frac{1}{2}}e^{-\frac{y^{2}}{2}},\quad y>0,
\end{equation}
(see \cite[(5.9), pp. 46]{ober}), setting  $y=(\sqrt{a})^{-1}r$ with  $a,r>0$, and considering the change of variable $s=\frac{x}{\sqrt{a}}$, then we obtain that

$$
\int_{0}^{\infty} (\sqrt{a}\:s)^{\mu+\frac{1}{2}}e^{-\frac{as^{2}}{2}}\sqrt{sr}\:J_{\mu}(sr)\sqrt{a}\:ds=\Biggl(\frac{r}{\sqrt{a}}\Biggr)^{\mu+\frac{1}{2}}
e^{-\frac{r^{2}}{2a}}
$$
so,
\begin{equation}\label{eqautov3}
\int_{0}^{\infty} s^{\mu+1}e^{-\frac{as^{2}}{2}}\:J_{\mu}(sr)\:ds=a^{-\mu-1}r^{\mu}
e^{-\frac{r^{2}}{2a}}\:,
\end{equation}
for all $a>0$.
Then,
$$\mathit{h}_{\mu }N_{\lambda }(y)=\int_{0}^{\infty}\lambda^{\frac{\mu}{2}}x^{\frac{1}{2}}\mathcal{K}_{\mu}
(\sqrt{\lambda}\:x)\sqrt{xy}\:J_{\mu}(xy)\:dx=$$
$$\int_{0}^{\infty}\lambda^{\frac{\mu}{2}}x^{\frac{1}{2}}\frac{1}{2}\Bigl(\frac{1}{2}\sqrt{\lambda}\:x\Bigr)^{\mu}\left[ \int_{0}^{\infty } e^{-t-\frac{\lambda x^{2}}{4t}}\frac{dt}{t^{\mu +1}}\right]\sqrt{xy}\:J_{\mu}(xy)\:dx=$$
\begin{equation} \label{eqHankelN}
\lambda^{\mu}\:y^{\frac{1}{2}}\Bigl(\frac{1}{2}\Bigr)^{\mu+1}\int_{0}^{\infty}x^{\mu+1}\left[ \int_{0}^{\infty }e^{-t-\frac{\lambda x^{2}}{4t}} \frac{dt}{t^{\mu +1}}\right]\:J_{\mu}(xy)\:dx.
\end{equation}
Applying de boundedness of function $z^{-\mu}J_{\mu}(z)$   we obtain that

$$\int_{0}^{\infty}x^{\mu+1}\left[ \int_{0}^{\infty }e^{-t-\frac{\lambda x^{2}}{4t}}\frac{dt}{t^{\mu +1}}\right]\:|J_{\mu}(xy)|\:dx=$$
$$y^{\mu}\int_{0}^{\infty}\left[ \int_{0}^{\infty }e^{-\frac{\lambda x^{2}}{4t}}x^{2\mu+1}\bigl|(xy)^{-\mu}J_{\mu}(xy)\bigr|\:dx \right]\:e^{-t}\:\frac{dt}{t^{\mu +1}}< \infty $$

\noindent  Then, reversing the orden of integration in (\ref{eqHankelN}) and  applying
(\ref{eqautov3}) we obtain  that

$$\mathit{h}_{\mu }N_{\lambda }(y)=\lambda^{\mu}\:y^{\frac{1}{2}}\Bigl(\frac{1}{2}\Bigr)^{\mu+1}\int_{0}^{\infty}\left[ \int_{0}^{\infty }x^{\mu+1}e^{-(\frac{\lambda}{2t})\frac{x^{2}}{2}}J_{\mu}(xy)\:dx\right]\:e^{-t}\frac{dt}{t^{\mu +1}}=$$
$$
\lambda^{\mu}\:y^{\frac{1}{2}}\Bigl(\frac{1}{2}\Bigr)^{\mu+1}\int_{0}^{\infty}\Bigl(\frac{\lambda}{2t}\Bigr)^{-\mu-1}y^{\mu}e^{-\frac{ty^{2}}{\lambda}} \: e^{-t}\frac{dt}{t^{\mu +1}}=\lambda^{-1}\:y^{\mu+\frac{1}{2}}\int_{0}^{\infty}e^{-t(1+\frac{y^{2}}{\lambda})} \: dt.
$$

\noindent Considering the change of variable  $s=t(1+\frac{y^{2}}{\lambda})$ in the last integral we obtain finally

$$\mathit{h}_{\mu }N_{\lambda }(y)=\lambda^{-1}\:y^{\mu+\frac{1}{2}}\int_{0}^{\infty}e^{-s} \Bigl(1+\frac{y^{2}}{\lambda}\Bigr)^{-1}\:ds=
\lambda^{-1}\:y^{\mu+\frac{1}{2}}\Bigl(1+\frac{y^{2}}{\lambda}\Bigr)^{-1}=\frac{y^{\mu+\frac{1}{2}}}{\lambda +y^{2}}.
$$
\end{proof}

\noindent Proof of Lemma  \ref{lemaHankelNlambda}.
\begin{proof}
Suppose that $f\in L^{p}(sr^{p})$ and $\psi \in \mathcal{H}_{\mu }$, we claim that

\begin{equation}\label{convoloc distr}
\int_{0}^{\infty }\left( N_{\lambda }\sharp f\right) (x)\psi
(x)dx=\int_{0}^{\infty }f(z)\left( N_{\lambda }\sharp \psi \right)
(z)dz.
\end{equation}
Indeed, we first observe that the following integral is finite
\begin{equation*}
\int_{0}^{\infty }\left| f(z)\right| \left[ \int_{0}^{\infty
}\int_{0}^{\infty }\left| N_{\lambda }(y)\right| \left| \psi (x)\right|
D_{\mu }(x,y,z)dxdy\right] dz
\end{equation*}
In fact, given a integer $q$ such that $\frac{1}{p}+\frac{1}{q}=1$,
the function
$$G(z)= \int_{0}^{\infty }\int_{0}^{\infty }\left|
N_{\lambda }(y)\right| \left| \psi (x)\right| D_{\mu }(x,y,z)dxdy$$
is in  $L^{q}(sr^{q})$ because it is the convolution of $\left|
N_{\lambda }(y)\right|\in L^{1}(sr)$ and $\left| \psi (x)\right|\in
L^{q}(sr^{q})$ ($\psi \in \mathcal{H}_{\mu })$. Since $f\in
L^{p}(sr^{p})$ then
\begin{equation*}
\int_{0}^{\infty }\left| f(z)\right| G(z)dz=\int_{0}^{\infty }(
r\left| f(z)\right| )( s^{-1}r^{-1}G(z))
s\:dz<\infty
\end{equation*}
because $r\left| f\right| \in L^{p}(s)$ and  $s^{-1}r^{-1}G = c_{\mu
}rG\in L^{q}(s)$. Then
\begin{eqnarray*}
\int_{0}^{\infty}f(z)\left(N_{\lambda }\sharp \psi \right)(z)\:dz
&=&\int_{0}^{\infty }f(z)\left[ \int_{0}^{\infty }\int_{0}^{\infty
}N_{\lambda }(y)\psi (x)D_{\mu }(x,y,z)dx\:dy\right] dz \\
&=&\int_{0}^{\infty }\left[ \int_{0}^{\infty }\int_{0}^{\infty
}f(z)N_{\lambda }(y)D_{\mu }(x,y,z)dz\:dy\right] \psi (x)dx \\
&=&\int_{0}^{\infty }( N_{\lambda }\sharp\:f) (x)\psi (x)dx
\end{eqnarray*}
and we thus get (\ref{convoloc distr}). The proof for $f\in
L^{\infty }(r)$ is similar.

\noindent Now, given $\phi \in \mathcal{H}_{\mu }$ and $f\in
L^{p}(sr^{p})$  or  $L^{\infty }(r)$,  by (\ref{convoloc distr}), we
have that
\begin{equation}\label{eqNlambdaf}
(\mathit{h}_{\mu }( N_{\lambda }\sharp \:f) ,\phi
)=(( N_{\lambda}\sharp\: f) ,\mathit{h}_{\mu }\phi
)=\int_{0}^{\infty} f(x)\:N_{\lambda }\sharp\: \mathit{h}_{\mu }\phi
(x)\: dx
\end{equation}

\noindent By Propositions \ref{prop hankel conv}, \ref{invHankel}  and item b) of  Lemma \ref{lemaNlambda}
\begin{equation*}
\mathit{h}_{\mu }( N_{\lambda }\sharp \mathit{h}_{\mu }\phi )(y) =r%
\mathit{h}_{\mu }( N_{\lambda }) \mathit{h}_{\mu }(\mathit{h}_{\mu }\phi)(y)=\frac{\phi (y)}{\lambda +y^{2}}.
\end{equation*}

\noindent So,
\begin{equation}\label{eqNlambdaH}
N_{\lambda }\sharp \:\mathit{h}_{\mu }\phi =\mathit{h}_{\mu }\Bigl(
\frac{\phi}{\lambda +y^{2}}\:\Bigr).
\end{equation}
Finally, from  (\ref{eqNlambdaf}) and (\ref{eqNlambdaH}) we obtain that for $\phi \in \mathcal{H}_{\mu }$ that

$$
(\mathit{h}_{\mu }( N_{\lambda }\sharp \:f) ,\phi
)=\int_{0}^{\infty} f(x)\:N_{\lambda }\sharp\: \mathit{h}_{\mu }\phi
(x)\: dx=\int_{0}^{\infty} f(x)\: \mathit{h}_{\mu }\Bigl(
\frac{\phi}{\lambda +y^{2}}\:\Bigr)(x)\: dx=$$
$$=\int_{0}^{\infty}\:\frac{1}{\lambda +x^{2}}\:\mathit{h}_{\mu }( f)(x)\phi(x)\:dx=\Bigl(\frac{1}{\lambda +x^{2}}\:\mathit{h}_{\mu }( f) ,\phi \Bigr)$$

\end{proof}

\vskip.4in
{\bf Acknowledgments}  Some of the main ideas of this paper were discussed with Miguel Sanz. The author wishes to thank him for the many helpful suggestions and for the stimulating conversations.

\end{document}